\documentclass[a4paper, 10pt, reqno]{amsart}
\usepackage{graphicx, color} 
\usepackage{amsmath, amsthm, amsfonts, amssymb, amscd}
\usepackage[T1]{fontenc}
\usepackage{comment}
\usepackage{amsmath}
\usepackage{tikz}
\usetikzlibrary{calc}
\usepackage{latexsym,amsfonts,amssymb,amsthm,euscript, subfigure}\usepackage{amsmath}
\usepackage{tikz}
\usetikzlibrary{calc}
\usepackage{latexsym,amsfonts,amssymb,amsthm,euscript, subfigure}
\usepackage{graphicx}
\usepackage{hyperref}
\usepackage{cleveref}
\usepackage{float}
\usepackage{color,fancybox}
\usepackage[normalem]{ulem}  
\usepackage{latexsym}
\usepackage{enumitem}
\usepackage{comment}
\usepackage[british]{babel}
\usepackage{indentfirst}

\usepackage{mathrsfs}

\theoremstyle{plain}
\newtheorem{theorem}{Theorem}[section]

\newtheorem{corollary}{Corollary}[section]
\newtheorem{definition}{Definition}[section]
\newtheorem{example}{Example}[section]

\newtheorem{lemma}{Lemma}[section]

\newtheorem{proposition}{Proposition}[section]
\newtheorem{remark}{Remark}[section]

\newtheorem{maintheorem}{Theorem}

\newtheorem{co}{Corollary}

\begin{document}

\title[Expanding measure has nonuniform specification property on RDS]{Expanding measure has nonuniform specification property on random dynamical system}

\author{R. Bilbao}
	\address{Rafael A. Bilbao, Escuela de Matem\'atica y Estat\'istica, UPTC, Sede Central del Norte Av. Central del Norte 39 - 115, cod. 150003 Tunja, Boyac\'a, Colombia} 
	\email{rafael.alvarez@uptc.edu.co}
	\urladdr{https://orcid.org/0000-0001-8223-9434}

	\address{Universidade Federal de Alagoas, Instituto de Matem\'atica - UFAL, Av. Lourival Melo Mota, S/N
	Tabuleiro dos Martins, Maceio - AL, 57072-900, Brasil} 

\subjclass[2010]{37A05, 37C05, 37C25}
\keywords{Random Dynamical Systems; Expanding measure; Nonuniform specification property }

\pagenumbering{arabic}

\begin{abstract}
In the present paper, we study the distribution of the return points in the fibers for a RDS (random dynamical systems) non uniformly expanding preserving an ergodic probability, we also show the abundance of nonlacunarity of hyperbolic times that are obtained along the orbits through the fibers. We conclude that any ergodic measure with positive Lyapunov exponents satisfies the nonuniform specification property between fibers. As consequences, we prove that any expanding measure is the limit of probability measure whose measures of disintegration on the fibers are supported by a finite number of return points and we prove that the average of the measures on the fibers corresponding to a disintegration, along the orbit $(\theta^n(w))_{n \geq 0}$ in the base dynamics is the limit of Dirac measures supported in return orbits on the fibers.
\end{abstract}

\maketitle
\section{Introduction}	
Given a dynamical system, we wish to know if there exist periodic points and understand their distribution on the underlying space. 
 In the following, we will mention some results in the deterministic dynamics. For a systems  $f:M \to M$  uniformly hiperbolic, Bowen established in \cite{Bowen2} that the asymptotic exponential growth of the set $P_{n}(f)$ of periodic points of period $n$ is determined by the topological entropy 

\[
\lim_{n \to \infty} \dfrac{\log P_{n}(f)}{n} = h_{top}(f).
\]

He introduced the important notion of \texttt{specification by periodic orbits} and proved a number of important results concerning the uniqueness and the ergodic properties of equilibrium states, asymptotic growth and the limit distribution of periodic orbits and so on, for Axiom A diffeomorphisms and flows \cite{Bowen2}. \\
There exists, various generalizations of the specification property among them are, almost specification, relative specification, periodic weak specification, weak specification, relative weak specification. The connection between them can be found in \cite{Dominik}. These kinds of specification given tools of great utility in exploring the topological structure and statistical behavior of some systems.

On the other hand, when we investigate the strong specification property of trajectories by periodic orbits, we realize that it is very strange to happen in the RDS (random dynamical system) context. Nevertheless, the authors Gundlach and Kifer in \cite{Kifer2} employ a random version that is based on approximation by arbitrary orbits and introduce notions of expansiveness, conjugation, specification for random bundle transformations and derive the uniqueness of equilibrium states for a large class of functions measures. \\
Problems related to the random dynamics of a skew product, where the dynamics in the fibers is non-uniform expanding have also been studied by J. Alves and V. Ara\'ujo \cite{Araujo}, taking random perturbations of nonuniformly expanding maps where they give sufficient conditions and necessary conditions for the stochastic stability of nonuniformly expanding maps either with or without critical sets. Arbieto, Matheus and Olivera \cite{Oliveira2} show the existence of equilibrium states for some random non-uniformly expanding maps. Another result, is the work of V. Ara\'ujo and J. Solano \cite{AraujoSolano}, they show the existence of absolutely continuous invariant probability measures for random one-dimensional dynamical systems with asymptotic expansion. More recently,  Stadlbauer, Suzuki and Varandas \cite{Stadlbauer-Suzuki-Varandas}  develop a quenched thermodynamic formalism for a wide class of random maps with non-uniform expansion, where no Markov structure, no uniformly bounded degree or the existence of some expanding dynamics is required and finally there is the work by  Bilbao and Ramos \cite{Alvarez2}, here it is considered a robust class of random non-uniformly expanding
local homeomorphisms and H\"older continuous potentials with small variation, proved that these equilibrium states and the random topological pressure vary continuously in this setting.
However, our point of view is quite different: we study a kind of \texttt{ weak specification property} for a non-uniform system class  on a skew product (see (\ref{specificationproper})).

\medskip
Back to the deterministic case, Oliveira uses this type of weak specification in \cite{Oliveira4} and proves that, given an endomorphism preserving an ergodic probability with positive Lyapunov exponents, there are periodic points of period growing sublinearly with respect to the length of almost every dynamical ball.  In particular, it implies that any ergodic measure with positive Lyapunov exponents satisfies the nonuniform specification property. 

{\bf Statements of the main result}. This article is motivated by the results of Oliveira \cite{Oliveira4} and we extend this weak specification property for random dynamical systems $F(w,x)= (\theta(w), f_{w}(x))$, where $\theta$ is an invertible map preserving an ergodic measure $\mathbb{P}$ and $f_w$ is a local homeomorphism of a compact Riemannian manifold exhibiting some non-uniform expansion. We say that $F$ has the {\it specification property between fiber} if there exist $\varepsilon_0 > 0$ such that for all $(w,x)\in X \times M$ and $0< \varepsilon < \varepsilon_0$, $n\geqslant 1$, there exists $(w,p)\in X \times M$ such that
\begin{itemize}
 \item $d(f^{i}_{w}(p),f^{i}_{w}(x))< \varepsilon$, for $i=0,1,...,n$
 \item $p$ has return less than $n+K_{\varepsilon}(n,(w,x))$ (see definition \ref{pontoperio})
\end{itemize}

\medskip
The system $F:X\times M \to X\times M$, will be conditioned to the hypotheses (I) - (V) and (H1)-(H3), will be stated in section 2 of preliminaries. The hypotheses (H1)-(H3) consists of regularity of the generating functions $f_w:M\to M$ with $w\in X$, which implies the existence of hyperbolic times on fibers. We obtain that this type of system satisfies the nonuniform weak specification property. For the skew product, it means that given a ball on fiber, there exists a point that returns after $n$ iterations, i.e. given $\varepsilon > 0$ and $B_{\varepsilon}(x) \subset M$, there exist  $p \in B_{\varepsilon}(x)\subset M$ and $ n\in \mathbb{N}$, such that $f^{n}_{w}(p)=p$ for all $w\in X$. In summary, the following results are obtained:

\begin{maintheorem}
 \label{theoremprincipal}
 Any $F-$ invariant expanding measure $\mu$ satisfies the nonunirform specification property. 
\end{maintheorem}
 
 A interesting consequence of Theorem A above, is following corollary, we prove that any expanding measure is the limit of invariant measure whose measures of disintegration on the fibers are supported by a finite number of return points.

\begin{co}
 \label{corola1}
 Given $\mu \in \mathcal{M}_{\mathbb{P}}^{1}(F)$ and $(\mu_w)_{w\in X}$ its respective disintegration, then each $\mu_w$ is accumulated in the weak$ ^\ast$ topology by probability measures $\nu_{w,r} = \sum_{j=1}^{r} k_{j,w} \delta_{p_{j,w}}$, with $\sum_{j=1}^{r} k_{j,w} = 1$ supported in a finite amount of return points $(p_{j,w})_{1\leq j \leq r}$ on the fiber  $M_w = \{x: \ (w,x)\in X\times M \}$.   Besides, given $\varphi \in C^0(M)$ the map $\displaystyle w\in X \mapsto \int_M \varphi  \ d\nu_{w,r}$ is measurable. Therefore, $\mu$ is accumulated in the weak$^{\ast}$ topology for the measures $\displaystyle \nu = \int \nu_{w,r} d\mathbb{P}(w)$. From here we can construct a $\tau$ $F$-invariant measure that is close to $\mu$ in the weak$^{\ast}$ topology.
\end{co}

A second consequence, shows that the average of the measures on the fibers corresponding to a disintegration, along an orbit in the base dynamics is the limit of Dirac measures supported in return orbits on the fibers.

\begin{co}
 \label{corola2}
 Given $F-$ invariant expanding ergodic measure $\mu$ and let $(\mu_w)_{w\in X}$ its desintegration in each fiber, then we have the following
 \begin{enumerate}
     \item For all $\varphi \in C^{0}(M)$, we have
\[
\left\lvert \int_{M} \varphi d\mu_{p,w,n+k} -  \int_{X}\int_{M}\varphi d\mu_{w}d\mathbb{P} \right\rvert < \dfrac{\varepsilon}{2},  \quad \mu-a.e.p..
\] 
where $ \displaystyle \mu_{p,x,n+k}=\frac{1}{n+k}\sum_{j=0}^{n+k-1}\delta_{f^{j}_{w}(p)}$,  with $f^{n+k}_{w}(p)=p$ and  $n$ large enough. 

\item
\[
\left \lvert \dfrac{1}{n}\sum_{j=0}^{n-1} \mu_{\theta^{j}(w)} -  \mu_{p,w,n+k} \right \rvert < \varepsilon, \quad \mu-a.e.p.
\]
for large $n$,  in the $weak^{\ast}$-topology. 

 \end{enumerate}
\end{co}

\section{Preliminaries}

Let $M$	be a compact and connected manifold with distance $d$ and $\Omega$ the space of local homeomorphisms defined on $M.$ Consider a Lesbesgue space $(X, \mathcal{A},\mathbb{P})$ and an invertible transformation $\theta:X\rightarrow X$ preserving $\mathbb{P}$. We call \emph{random dynamical system} (RDS) any continuous transformation $f:X\to \Omega$ given by $w\mapsto f_w\in \Omega$ such that $(w, x)\mapsto f_w(x)$ is measurable. For every $n\geq 0$ we define 
\[
f^0_w:=Id\ \ , \ \ f^{n}_{w}:=f_{\theta^{n-1}(w)}\circ...\circ f_{\theta(w)}\circ f_{w}\ \ \mbox{and}\ \ f^{-n}_{w}=(f^{n}_{w})^{-1}.
\]
The skew-product generated by the maps $f_w$ is the measurable transformation 
$$F:X\times M \to X\times M \ \ ; \ \ F(w, x)=(\theta(w),f_{w}(x)).$$
In particular, $F^n(w, x)=(\theta^n(w), f^{n}_{w}(x) )$ for every $n\in \mathbb{Z}.$

 Let $\mathcal{M}^{1}_{\mathbb{P}}(X \times M)$ be the collection of all probability measures $\mu$ on $(X \times M, \mathcal{B})$ such that
 \[
  \mu\circ \pi^{-1}_{X}=\mathbb{P} 
 \]
 where $\pi_X:X \times M \rightarrow X$ is the first projection ($\pi(x,y)=x$).
 
 \medskip
 Denote by $\mathcal{M}_{\mathbb{P}}(F)\subset\mathcal{M}^{1}_{\mathbb{P}}(X \times M)$ the set of $F$-invariant measures. Let $\varepsilon_{X}$  the partition on $X$ into singletons and $\pi^{-1}_{X}(\varepsilon_{X})$ the partition whose elements are the fibers. We have that they are measurable. Notice that, by Rokhlin's disintegration theorem \cite{Rokhlin1}, for every $\mu\in \mathcal{M}_{\mathbb{P}}(F)$ there exists a system of sample measures $\{\mu_{w}\}_{w\in X}$ of $\mu$ such that  $$d\mu(w, x)=d\mu_w(x)\ d\mathbb{P}(w).$$
 We say that a $F$-invariant measure $\mu$ is \emph{ergodic} if $(F, \mu)$ is ergodic. In what follows we assume that the system $(\theta, \mathbb{P})$ is ergodic. 
 
\medskip
Let $\mathcal{C}_{\mathbb{P}}(X\times M)$ the space of all measurable functions  $\phi:X\times M \rightarrow \mathbb{R}$ such that $\phi_{w}:M\rightarrow \mathbb{R}$ defined by $\phi_{w}(x)=\phi(w,x)$ is continuous for all $w\in X$. We further define 
\[
 \mathbb{L}^{1}_{\mathbb{P}}(X,C^0(M)):=\biggl\{\phi\in \mathcal{C}_{\mathbb{P}}(X\times M): \lVert \phi\rVert_{1}=\int_{X}\lVert \phi_{w}\rVert_{\infty}\,d\mathbb{P}(w)<+\infty\biggr\}
\]

\medskip

{\bf Hypothesis about the generating maps $f_w$.}  For each $w\in X$ let $f_w:M \rightarrow M$ be a local homeomorphism satisfying: there exists a continuous function $L_w :M\rightarrow\mathbb{R}$ such that for every $x \in M$ we can find a neighborhood $U_x$ where $f_w: U_x \rightarrow f_w(U_x)$ is invertible and $$d(f_w^{-1}(y), f_w^{-1}(z)) \leq L_{w}(x)d(y, z), \ \ \mbox{for all}\ \ y,z \in f_w(U_x).$$
Notice that, the number of preimages $\#f_w^{-1}(x)$ is constant for all $x \in M$. We set $\deg(f_w):=\#f_w^{-1}(x)$ the degree of $f_w$ and assume $\deg(F)= \sup_w \deg(f_w) < \infty$.	

We suppose that there exists an open region $\mathcal{A}_w\subset M$ and constants $\sigma_w>1$ and $L_w >1$ close enough to $1$ such that  
\begin{enumerate}
	\label{cond0}
	\item [(I)] $L_{w}(x) \leq L_w$ for every $x \in \mathcal{A}_w$ and $L_{w}(x) < \sigma_w ^{-1}$ for every $x \in \mathcal{A}_{w}^{c}=M\setminus \mathcal{A}_w$.
	\item[(II)] There exists a finite covering $\mathcal{U}_w$ of $M$, by open domains of injectivity for $f_w$, such that $\mathcal{A}_w$ can be covered by $q_w < \deg(f_w)$.
	\item [(III)] For every $\varepsilon>0$ we can find some positive integer $\tilde{n}=\tilde{n}(w,\varepsilon)$ satisfying $f_{\theta^j(w)}^{\tilde{n}}(B_{\theta^j(w)}(f^{j}_{w}(x),\varepsilon))=M$ for any $j\geq 0.$
	\item [(IV)] We assume the existence of some positive $\varepsilon_0 >0$ satisfying for all $w\in X$ the following
	\begin{equation}
	    \label{hipo4}
	    e^{\varepsilon_0} < \dfrac{deg f_w}{q_w}
	\end{equation}

	\item [(V)] 	Suppose that there exists $c>0$ such that for $\mathbb{P}$-almost every $w\in X$ we can find $\hat{L}_w$ close enough to $1$ and $\hat{\sigma}_w >1$ satisfying for every $j \geq 0$ that
	\begin{equation}
	  \label{condV}  
	L_{\theta^{j}(w)} \leq \hat{L}_w \quad, \quad \hat{\sigma}_w \leq \sigma_{\theta^{j}(w)} \quad \text{and} \quad \hat{L}_{w}^{\rho}\hat{\sigma}_{w}^{-(1-\rho)}< e^{-2c} < 1
	\end{equation}
	where $\rho$ is given by Lemma \ref{lemma1}.
	\end{enumerate}
 The conditions (I) and (II) mean that it is possible the existence of expanding and contracting behavior in $M$ but it is required for every point at least one preimage in the expanding region. The condition (III) means that the skew-product $F$ is \emph{topologically exact}.
\medskip

Moreover, we assume that:
\begin{enumerate}
    \item [(H1)]
    \begin{align*}
        \label{hipo1}
         \int \log^{+}  L_{w}(x) d\mu(w,x)< +\infty \ \  \text{and} \ \ \int \log^{+}  L_{w}(x)^{-1} d\mu(w,x)< +\infty
    \end{align*}
    for each $\mu \in \mathcal{M}^{1}_{\mathbb{P}}(F)$.
\item [(H2)]

 \label{hypo2}

 $\displaystyle\inf_{(w,x)\in X\times M} L_{w}(x) > 0$

\item [(H3)]
For $\varepsilon>0$, there exist $\delta>0$ such that for $x\in M$ and $z\in B_{\delta}(x)\subset M$, we have

\begin{equation}
 \frac{ L_{w}(x)}{ L_{w}(z)}\leq \exp (\varepsilon/2)  
\end{equation}
 $\mathbb{P}$-a.e. $w\in X$.
\end{enumerate}

Here we establish the definition of return point between fibers and the nonuniform specification property. 

\begin{definition}
 \label{pontoperio}
 A point $x\in M$ is a return point between fibers if there exist $w \in X$ and $n\in \mathbb{N}$ such that $f_{w}^{n}(x)=x$. In the case that $\theta^{n}(w)=w$, then $(w,x)$ is a periodic point of $F$.
\end{definition}

\medskip

On the other hand, a $(n, \varepsilon)$ {\it uniform dynamical ball} is defined as
\[
B_{n,w}(x, \varepsilon)= \bigcap_{k=0}^{n}f^{-k}_{w}\Bigl(B_{\theta^{k}(w)}\bigl(f^{k}_{w}(x), \varepsilon \bigr)\Bigr)\subset M.
\]
 
We say that $(F,\mu)$ has the {\bf  nonuniform specification property between fibers} if for $\mu-$ almost everywhere, the ball $B_{n,w}(x, \varepsilon)$ contains a return point whose ``period'' is less than $n+ K_{\varepsilon}(n,(w,x))$ satisfying

\begin{equation}
\label{specificationproper}
   \limsup_{n\rightarrow \infty} \frac{K_{\varepsilon}(n,(w,x))}{n}=0.  
\end{equation}

\medskip

Let $F: X\times M\rightarrow X\times M$ be a RDS and $\mu$  a probability. 
We say that a real, nonnegative, integrable function $J_\mu (F)$ is a Jacobian if there exists $\mathcal{I} \subset X\times M$ of measure
0 such that for every measurable set $A \subset X\times M\setminus \mathcal{I}$ on which $F$ is injective we have that 
$$
\mu(F(A))=\int_{A}  J_{\mu}(F)\, d\mu.
$$

 Restricting $J_\mu (F)$ to $M_w=\{x \ : \ (w,x)\in X\times M \}= M$ and neglecting the zero measure set $\mathcal{I}$, we may  consider the function $J_{\mu_w}(f_w) (x) = J_\mu (F)(w,x)$ on the fiber $M_w$, for $\mathbb{P}$-almost every $w \in X$. It is clear from the definition above that :
$$
\mu_{\theta(w)}(f_w(A_w))=\int_{A_w}  J_{\mu_w}(f_{w})\, d\mu_w,
$$
For any $A_w \subset M_w \setminus \mathcal{I}_w $   measurable such that and  $f_w|_{A_w}$ is injective. In particular, $J_{\mu_w}$ is the jacobian of $f_w$ relative to $\mu_w$.

In our setting, since $F$ is essentially countable  to one,  Proposition 1.9.5 of \cite{Urbanski3} give us that  Jacobians of any measure $\mu$ exists. 

\medskip
We recall that the {\it support} of an invariant measure $\mu$ is the full measure set $supp(\mu)$ of all points such that any neighbourhood of each one of them has positive measure.

\begin{lemma}
		\label{lemmajacobiano}
		 $supp(\mu_w)=M$ if $F$ is topologically exact.
	\end{lemma}

\begin{proof}
	Now suppose that $F$ is topologically exact.  Given $U_w$ some non-empty open set $M$ such that $\mu_{w}(U_w)=0$, by the exactness assumption, we can take  $N_w\in \mathbb{N}$ such that $f^{N_w}_{w}(U_w)=M$. Partitioning $U_w$ into mensurable subsets $U_{w,1},...,U_{w,k}$ such that every $f_{w}^{N_w}|_{U_{w,j}}$ is injective for $j=1,...,k$. we get that
		$$
		\mu_{\theta^{N_{w}}(w)}\bigl( f^{N_w}_{w}(U_{w,j})\bigr)=\int_{U_w,j}J_{\mu_w}f^{N_w}_{w} d\mu_{w}=0.
		$$
		Hence	$$	\mu_{\theta^{N_{w}}(w)}(M)\leq\sum^{k}_{j=1}\mu_{\theta^{N_{w}}(w)}\bigl( f^{N_w}_{w}(U_{w,j})\bigr)=0$$
		which is a contradiction. This completes the proof.
	
\end{proof}

\section{Hyperbolic times}

For the next step we need the notion of hyperbolic times, as in \cite{Alves1}. But first we need the following definition. 
\medskip
\begin{definition}
\label{nonuniformlyexp}
 We say that a measure $\mu \in \mathcal{M}_{\mathbb{P}}(X \times M)$ is expanding with exponent $c$ if for $\mu-$almost every $(w,x)\in X \times M$ we have:

  \begin{equation}
     \label{measurexpanding}
   \limsup_{n\rightarrow \infty} \frac{1}{n}\sum^{n-1}_{j=0}\log L_{\theta^{j}(w)}(f^{j}_{w}(x))\leqslant -2c<0
 \end{equation}

\end{definition}

\begin{definition}
 \label{timehyp}
Given $c>0$, we say $n\in \mathbb{N}$ is a $c$- hyperbolic time of $F$ for  $(w,x)\in X \times M$ if 
 \[
  \prod^{n-1}_{j=n-k}L_{\theta^{j}(w)}(f^{j}_{w}(x)) \leqslant\exp(-ck), \quad \text{for every} \ 1\leqslant k\leqslant n
 \]
 
\end{definition}

We say that $F$ has \textit{positive density of hyperbolic times} for $(w,x)$, if the set $H_{(w,x)}$ of integers which are hyperbolic times of $F$ for $(w,x)$ satisfies
\begin{equation}
\label{defdenpositive}
 \liminf_{n\rightarrow +\infty}\frac{1}{n}\# \big(H_{(w,x)}\cap [1,n]\big)>0
\end{equation}

\begin{definition}
\label{tempohiper}
We consider $H_{n}(c,\delta, F)$ the set of points $(w,x)$ such that $n$ is a hyperbolic time of $(w,x)$ and define 
\[
H(c,\delta, F):= \bigcap_{l\geqq 1} \bigcup_{n \geqq l}H_{n}(c,\delta, F)
\]
the set points with infinitely many $(c,\delta)$-hyperbolic times for $F$.
\end{definition}

\medskip

As shown in \cite{Oliveira1}, we  have infinity many hyperbolic times for expanding measures. For this we need of Pliss Lemma \cite{Alves1}

\begin{lemma}[Pliss]
\label{lemPliss}
Let $0<c_1<c_2<A$ and $\zeta=\dfrac{c_{2}-c_{1}}{A-c_{1}}$. Given real numbers $a_1,...,a_{N_0}$ satisfying $a_j\leq A$ for each $1\leq j\leq N_0$ and
\[
 \sum_{j=1}^{N_0}a_{j}\geq c_{2}N_0,
\]
there exist $l>\zeta N_0$ e $1<n_1<\cdot\cdot\cdot<n_l\leq N_0$ such that
\[
 \sum_{j=n+1}^{n_i}a_{j}\geq c_{1}(n_{i}-n)
\]
for each $0\leq n<n_{i}$ and $i=1,...,l$.
\end{lemma}

This lemma of Pliss together with (H2), we can find a set of full measure $H$, such that all its points has infinitely many hyperbolic times, which can be seen in the following lemma.

\begin{lemma}
\label{leminfth}
For every invariant expanding measure $\mu$ with $c-$expanding, there exists a full $\mu-$measure set $H\subset X \times M$ that every $(w,x)\in H$ has infinitely many hyperbolic times $n_{i}=n_{i}(w,x)$ with exponent $c$ and, in fact, the density of hyperbolic times at infinity is larger than some  $\alpha=\alpha(c)>0$:
\begin{itemize}
 
 \item  $ \displaystyle \prod_{j=n_{i}-k}^{n_{i}-1} L_{\theta^{j}(w)}(f^{j}_{w}(x))\leq \exp(-ck)$ for every $1\leq k \leq n_{i}$
 
 \item  $\displaystyle \liminf_{n\rightarrow \infty} \dfrac{\#\{0\leq n_{i}\leq n\}}{n}\geq \alpha>0$
\end{itemize}
\end{lemma}


As a consequence of the previous lemma and the distortion hypothesis (H3), we obtain that.

\begin{lemma}
\label{lemacontra}
There exist $\delta>0$ such that for $\mathbb{P}$-almost every point $w\in X$, if $n_{i}$ is a hyperbolic time of $(w,x)$ and let $z_{n_i}\in B_{\delta}(f^{n_i}_{w}(x))$, then there exists $z\in B_{\delta }(x)\subset M$ such that $f^{n_i}_{w}(z)=z_{n_i}$ and
$$
d(f^{n_{i}-k}_{w}(z), f^{n_{i}-k}_{w}(x))\leq \exp(-ck/2)d(f^{n_{i}}_{w}(z),f^{n_{i}}_{w}(x)), 
$$
for each $1\leq k\leq n_i$.
\end{lemma}

The proofs of Lemmas \ref{leminfth} and \ref{lemacontra} above may be found in \cite[Lemmas 5.4 and 5.5]{Oliveira2}.

\section{Transfer Operator }
For $w\in X$, let  $f_w:M\rightarrow M$ be the fiber dynamics and $\phi_w:M\rightarrow \mathbb{R}$ be the potential. Consider $\mathcal{L}_{w}:C^0(M)\rightarrow C^0(M)$ the transfer operator	associated to $(f_w, \phi_w)$ defined by
$$\mathcal{L}_{w}(\psi)(x)=\sum_{y\in f_w^{-1}(x)}e^{\phi_w(y)}\psi(y).$$
Consider also its dual operator $\mathcal{L}_{w}^{\ast}:[C^0(M)]^\ast\to[C^0(M)]^\ast $ which satisfies
$$\int \psi\, d\mathcal{L}_{w}^{\ast}(\rho_{\theta(w)})=\int\mathcal{L}_{w}(\psi)\, d\rho_{\theta(w)}.$$

We say that a probability measure $\mu_w\in \mathcal{M}^1(M)$ is a \emph{reference measure} associated to $\lambda_w\in\mathbb{R}$ if $\mu_w$ satisfies 
$$\mathcal{L}_{w}^{\ast}(\mu_{\theta(w)})=\lambda_w\mu_w.$$

As in the deterministic case, by applying the Schauder-Tychonoff fixed point theorem, it is straightforward to prove the existence of a system of reference measures $\{\mu_w\}_{w\in X}$ where $\mu_w$ is associated to $\lambda_w$ given by
\begin{equation}
\label{lambda}
\lambda_{w}=\mathcal{L}^{\ast}_{w}\mu_{\theta(w)}(1)=\mu_{\theta(w)}(\mathcal{L}_{w}(1))
\end{equation}
for $\mathbb{P}$-almost every $w\in X.$ See \cite{Urbanski2} for details. In the sequel we derive some properties of the reference measure. For the case where the potential $\phi_w \equiv 0$ for all $w\in X$, then $\lambda_w = deg(f_w)$ for all $w\in X$.

\begin{lemma}
	\label{lemmajacobiano}
	The jacobian of $\mu_w$ with respect to $f_w$ is given by $J_{\mu_w}f_w = \lambda_{w}e^{-\phi_w}$. 
\end{lemma}

\begin{proof}
	Let $A\subset M$ be a measurable set such that $f_{w}|_{A}$ is injective.  Take a bounded sequence $\{\zeta_n \}\in C^0(M)$ such that $\zeta_{n}\to \mathcal{X}_{A}$. Then
	\begin{eqnarray*}
		\int_{M} \lambda_{w}e^{-\phi_w}\zeta_{n} d\mu_{w}&=&\int_{M} e^{-\phi_w}\zeta_{n} d(\mathcal{L}^{\ast}_{w}\mu_{\theta(w)})=\int_{M} \mathcal{L}_{w}(e^{-\phi_w}\zeta_{n})(y) d\mu_{\theta(w)}(y)\\
		&=& \int_{M} \sum_{f_{w}(z)=y}\zeta_{n}(z)d\mu_{\theta(w)}(y)=\int_{M} \sum_{f_{w}(z)=y} \zeta_{n}(f^{-1}_{w}(y))d\mu_{\theta(w)}(y).
	\end{eqnarray*}
	Since $\displaystyle \int_{M} \sum_{f_{w}(z)=y} \zeta_{n}(f^{-1}_{w}(y))d\mu_{\theta(w)}(y)  \longrightarrow \int_{M} \mathcal{X}_{A}(f^{-1}_{w}(y))d\mu_{\theta(w)}(y)$ when $n\to\infty$
	and $\displaystyle \int_{M} \mathcal{X}_{A}(f^{-1}_{w}(y))d\mu_{\theta(w)}(y)=\int_{M}\mathcal{X}_{f_{w}(A)} d\mu_{\theta(w)}=\mu_{\theta(w)}(f_{w}(A)).$
	We conclude that
	$$	\mu_{\theta(w)}(f_{w}(A))=\int_{A} \lambda_{w}e^{-\phi_w} d\mu_{w}.$$
	Notice that, by induction, we have for every $n\in\mathbb{N}$ that	
	\begin{equation}
	\label{equaintlocal}
	\mu_{\theta^{n}(w)}(f^{n}_{w}(A))=\int_{A}\lambda^{n}_{w}e^{-S_{n}\phi_{w}} \ d\mu_{w},
	\end{equation}
	where $\lambda^{n}_{w}:= \lambda_{\theta^{n-1}(w)} \lambda_{\theta^{n-2}(w)}\cdot \cdot \cdot \lambda_{\theta(w)}\lambda_{w}$.

\end{proof}

\medskip
Consider $c>0$ given by condition~(\ref{measurexpanding}). Given $w\in X$, let $H_w\subset M$ be the subset of $M$ such that $(w, x)$ has infinitely many hyperbolic times, i.e.,

$$
H_w:=\left\{x\in M\  \big|\ \limsup_{n\to+\infty} \frac{1}{n}\sum^{n-1}_{j=0}\log L_{\theta^{j}(w)}(f^{j}_{w}(x))\leqslant -2c<0\right\}.
$$
Next we prove that $\mu_w(H_w)=1$ for $\mathbb{P}$-almost every $w\in M.$

\medskip

Let $\mathcal{P}$ be a partition of $M$  with cardinality $\# \mathcal{P} = k$. We suppose without loss of generality that the set $\mathcal{A}_w$ is contained in the first $q_w$ elements of $\mathcal{P}$ for all $w\in X$. Consider the numbers $$\bar{p}_w = k - q_w, \,\,\,\hat{q} = \sup_{w \in X} q_{w},\,\,\, \bar{q}= \inf_{w\in X} q_{w}\,\,\,  \mbox{and}\,\,\, \hat{p} = \sup_{w\in X} \bar{p}_{w}.$$
This numbers are well defined since we assume that $\deg(F)= \sup_w \deg(f_w) < \infty$.

For $\rho \in (0,1)$ and $n\in\mathbb{N}$ let $I(\rho, n)$ be the set of itinerates 
$$I(\rho, n)\!=\!\{(i_w,..., i_{\theta^{n-1}(w)})\in \{1,...,k \}^{n};\# \{0 \leq j \leq n-1 : i_{\theta^{j}(w)}\leq q_{\theta^{j}(w)} \} > \rho n \}$$
and consider
$$C_\rho := \limsup_{n}\dfrac{1}{n} \log \# I(\rho, n).$$

\begin{lemma}\cite[Lemma 3.1]{Varandas-Viana}
	\label{lemma1}
	Given $\varepsilon > 0$ there exists $\rho_0 \in (0,1)$ such that $C_\rho < \log \hat{q} + \varepsilon$ for every $\rho \in (\rho_0 , 1)$. 
\end{lemma}

\begin{proof}Notice that
	$  \displaystyle \# I(\rho, n) \leq \sum_{k = [\rho n]}^n \binom{n}{k} q_w q_{\theta(w)}\cdot \cdot \cdot q_{\theta^{k-1}(w)}p_w p_{\theta(w)}\cdot \cdot \cdot p_{\theta^{n -(k-1)}(w)}.
	$
	By applying Stirling's formula we have
	\[
	\sum_{k=[\rho n]}^{n} \binom{n}{k} = \dfrac{n}{2}\binom{n}{[\rho n]} \leq C_1 \exp{(2t(1 - \rho)n)}, \quad \mbox{for} \quad \rho > \dfrac{1}{2}.
	\]
	Thus there exist $C_1$ and $t > 0$ such that $
	\# I(\rho, n) \leq C_1 \exp{(2t(1 - \rho)n)} \ \hat{q}^n \ \hat{p}^{(1-\rho)n}.
	$
	Taking the limit when $n$ goes to infinity we have $$ C_\rho = \limsup_n \frac{1}{n}\log \# I(\rho, n) \leq \log \hat{q} + \varepsilon
	$$
	for any $\rho$ close enough to 1.
\end{proof}

From this lemma we can fix of $\rho < 1$ such that $$C_\rho < \log \hat{q} + \dfrac{\varepsilon_0}{4}.$$ 
Recalling the definition of $\lambda_w$ in ($\ref{lambda}$) and the equation (\ref{hipo4}) we have that
\begin{align*}
\lambda_w \geq \deg f_w \geq e^{ \log q_w + \varepsilon_0}.
\end{align*}
Now, using Lemma~\ref{lemmajacobiano}  and considering $\phi_w \equiv 0$ for all $w\in X$, we obtain that
\begin{equation}
\label{equadejacobiano1}
J_{\mu_w}f_w = \lambda_w \geq e^{\log q_w + \varepsilon_0} > q_w .
\end{equation}

\begin{proposition} \label{expanding} We have $\mu_w(H_w)=1$ for a.e. $w\in X$.
\end{proposition}
\begin{proof}
	Given $n\in\mathbb{N}$ denote by $B_w(n)$ the set of points $x\in M$ whose frequency of visits to $\{\mathcal{A}_{\theta^{j}(w)}\}_{0\leq j \leq n-1}$ up to time $n$ is at last $\rho$, that means, 
	\[
	B_w(n) =\left\{x\in M \ | \ \dfrac{1}{n} \# \{0 \leq j \leq n-1 : \ f^{j}_{w}(x) \in A_{\theta^j (w)}  \} \geq \rho  \right\}.
	\]
	Let $\mathcal{P}^{(n)}$ be the partition $\bigvee_{j=0}^{n-1}(f_w^j)^{-1}\mathcal{P}$. We cover $B_w(n)$ by elements of $\mathcal{P}^{(n)}$ and since $f^n_w$ is injective on every $P\in \mathcal{P}^{(n)}$, we may use (\ref{equadejacobiano1}) to obtain
	\begin{align*}
	1 \geq \mu_{\theta^n(w)}(f^n_w (P)) &= \int_P J_{\mu_w}(f^n_w) \ d\mu_w = \int_P \prod_{j=0}^{n-1}J_{\mu_{\theta^{j}(w)}}f_{\theta^j (w)} \ d\mu_w\\
	& \geq \prod_{j=0}^{n-1} e^{(\log q_{\theta^j(w)} + \varepsilon_0)}\mu_w (P) \geq e^{(\log \bar{q} + \varepsilon_0)n}\mu_w (P).
	\end{align*}
	Thus	
	$$
	\mu_w (P) \leq e^{-(\log \bar{q} + \varepsilon_0)n}.
	$$
	Since we can assume that $\hat{q}< \bar{q}e^{\varepsilon_0 /2}$ for every $w\in X$ we have  
	\begin{align*}
	\mu_w (B_w(n)) & \leq \# I(\rho,n) e^{-(\log \bar{q} + \varepsilon_0)n} \\
	& \leq e^{(\log \hat{q} + \varepsilon_0/4)n}e^{-(\log \bar{q} + \varepsilon_0)n}
	\leq e^{(\log \hat{q}/\bar{q} - \varepsilon_0/2)n}.
	\end{align*}
	Hence the measure $\mu_w( B_w(n))$ decreases exponentially fast when $n$ goes to infinity. 
	Applying the Borel-Cantelli lemma we conclude that $\mu_w$-almost every point belongs to $B_w(n)$ for at most finitely many values of $n$. Then, in view of our choice (\ref{condV}) we obtain for $n$ large enough that 	
	\[
	\sum_{j=0}^{n-1}\log L_{\theta^j(w)}(f^{j}_{w}(x)) \leq \rho \log \tilde{L}_w + (1-\rho)\log \tilde{\sigma}^{-1}_w \leq -2c < 0
	\]
	which proves that $\mu_w$-almost every point has infinitely many hyperbolic times.
\end{proof}
Notice that from the last proposition and recalling that $\mu_w$ is an open measure we conclude that $H_w$ is dense in $M$.	

\medskip
Define the {\it first hyperbolic time} function $n_{1}: H(c, \delta, F)\rightarrow \mathbb{N}$ setting $n_{1}(w, x)$ as the first hyperbolic time of $(w,x)$. Next we show that $n_{1}$ is integrable.

\begin{corollary}
 \label{n1integrable}
 $$
 \int n_1 d\mu < \infty
 $$
\end{corollary}

\begin{proof}
Using that for every $w\in X$ the measure of $B_w (n)$ decreases exponentially fast when $n$ goes to infinity. We get,

\begin{align*}
\int_{X\times M} n_1 d\mu = \int_{X}\int_{M} n_1(w,x)d\mu_{w}d\mathbb{P}& = \int_X \sum_{n=0}^{\infty} \mu_{w}(\{x: n_{1}(w,x)>n  \})d\mathbb{P}\\
& \le \int_{X} \left(1 + \sum_{n=1}^{\infty} \mu_{w}(B_{w}(n)) \right)d\mathbb{P} \\
& \leq \int_{X} \left(1 + \sum_{n=1}^{\infty}  e^{(\log \hat{q}/\bar{q} - \varepsilon_0/2)n} \right)d\mathbb{P}\\
& = 1 + \sum_{n=1}^{\infty}  e^{(\log \hat{q}/\bar{q} - \varepsilon_0/2)n} < \infty
\end{align*}
\end{proof}

\medskip
Let $\gamma :\mathbb{R}^{+}\to \mathbb{R}^{+}$ a bijection. We say that a increasing sequence $(a_k)_{k\in \mathbb{N}}$ is $\gamma$-nonlacunary, if
\[
\lim_{k\rightarrow \infty}\frac{a_{k+1}-a_k}{\gamma(a_k)}=0
\]
\begin{remark}
 \label{diferenciatemhiper}
 Before prove the Lemma follow, we remark that growth rate of $n_{i+1}-n_{i}$ is obtained by analysing the decay of measure of
 \[
  Q_{n}=H_{n}\setminus \bigl(\bigcup_{j<n}H_{j}\bigr)= \{ (w,x)\in X \times M: n_{1}(w,x)=n\}
 \]

\end{remark}

\begin{lemma}
\label{lmlacunary}
If $\gamma^{-1}$ denotes the inverse function of $\gamma$, assume that for every $r>0$ the function $\gamma^{-1}\circ (rn_1)$ is $\mu$-integrable. Then, for $\mu$-a.e. $(w,x)\in X\times M$ the sequence $n_{j}( \ )$ of its hyperbolic times is $\gamma$-nonlacunary.
\end{lemma}

\begin{proof}
 Let $D$ be the set of point for which the sequence $n_{j}( \ )$ fails to be $\gamma-$ nonlacunary. For each $r>0$, define $L_{r}(n):=\{ (w,x)\in X \times M: n_{1}(w,x)\geq r\gamma (n)\}$. If $(w,x)\in D$, then there exists a rational number $r>0$, and there are infinitely many values of $i$ such that $n_{i+1}(w,x)-n_{i}(w,x)\geq r \gamma(n_{i}(w,x))$. This implies that
 \[
  n_{1}((F)^{n_i}(w,x))=n_{i+1}(w,x)-n_{i}(w,x)\geq r\gamma(n_{i}(w,x))
 \]
So, there are arbitrarily large values of $n$ such that $(w,x)\in (F)^{-n}(L_{r}(n))$. In other words, $D$ is contained in the set
\[
 L=\bigcup_{r\in \mathbb{Q}\cap (0, +\infty)}\bigcap_{k=0}^{\infty}\bigcup_{n\geq k}(F)^{-n}(L_{r}(n))
 \]
Since $\mu$ is invariant, we have $\mu((F)^{-n}(L_{r}(n)))=\mu(L_{r}(n))$ for all $n$. Then

\begin{align*}
 \sum_{n=1}^{\infty}\mu(L_{r}(n))= \sum_{n=1}^{\infty}\sum_{n_{1}\geq r\gamma(n)}\mu(Q_{n_1})&=\sum_{m=1}^{\infty}\sum_{n=1}^{[\gamma^{-1}(m/r)]}\mu(Q_m)\\
 &\leq \sum_{m=1}^{\infty}\gamma^{-1}(m/r)\mu(Q_m)
\end{align*}
Thus, using the hypothesis that $\gamma^{-1}(n_{1}/r)$ is integrable,
\[
 \sum_{n=1}^{\infty}\mu(L_{r}(n))\leq \sum_{m=1}^{\infty}\gamma^{-1}(m/r)\mu(Q_m)=\int \gamma^{-1}(n_{1}(w,x)/r)d\mu(w,x)<\infty.
\]
By the Borel-Candelli Lemma, this implies that $L$ has measure zero. It follows that $\mu(D)=\mu(L)=0$, as claimed.

\end{proof}

\begin{corollary}
 \label{seqtempononlacu}
 If $\mu$ is a invariant expanding ergodic measure, then the sequence of hyperbolic times is nonlacunary.
\end{corollary}

\begin{proof}
 Observe that $n_{1}$ is integrable by Corollary \ref{n1integrable}. To finish the proof, just put $\gamma(t)=t$ in Lemma \ref{lmlacunary}.
\end{proof}

\section{Proof of theorem A}

\begin{lemma}
 \label{exispontoperio}
 For $\mu$-a.e. $(w,x)\in X\times M$ we have that given $\varepsilon >0$ small enough, if $B_{n, w}(x, \varepsilon)$ is a dynamical ball at $x$ for $f_{w}$, then there exists a return point in $B_{n, w}(x, \varepsilon)$, with ``period'' less than $n+ K(n, (w,x), \varepsilon)$ where 
 \[
  \limsup_{n \rightarrow \infty}\frac{K(n,(w,x),\varepsilon)}{n}=0.
 \]

\end{lemma}

\begin{proof}
Given $(w,x)\in X \times M$ with nonlacunary infinite sequence $(n_k)_k$ of $F-$hyperbolic times (Lemma \ref{lmlacunary} and Corollary \ref{seqtempononlacu}).We can assume that $(w,x)$ is such a point on $Supp(\mu)$, the support of $\mu$. Let $\varepsilon >0$ and $n>1$ be fixed with $\varepsilon < \delta$(given in Lemma \ref{lemacontra}) and let $k$ be such that $n_{k} < n \leq n_{k+1}$.

 \medskip
 For $0\leq h\leq n_{k+1}$, there exists $\gamma_{\varepsilon}:X\rightarrow \mathbb{R}$ with $\gamma_{\varepsilon}:=\gamma_{\varepsilon}(w)<\varepsilon$, using the uniform of the $f_w$, we can to define, for all $x\in M$, with $\gamma_{w}=\min_{0\leq h\leq n_{k+1}} \{r_{h}: \ f^{h}_{w}(B_{w}(x,r_{h}))\subset B_{\theta^{h}(w)}(f^{h}_{w}(x),\varepsilon)\}$. Therefore
 \[
  f^{h}_{w}(B_{w}(x, \gamma_{x}))\subset B_{\theta^{h}(w)}(f^{h}_{w}(x),\varepsilon)
 \]

\medskip
Let $V(z)=f_{w}^{-n_{k+1}}(B(f_{w}^{n_{k+1}}(z),\gamma_{w}))$ be the $f_{w}-$hyperbolic pre-ball around $z$ of length $n_{k+1}$. Since $n_{k+1}$ is a $f_{w}$-hyperbolic time for $z$ and $f_{w}^{-j}$ is a contraction on $B(f_{w}^{n_{k+1}}(z),\gamma_{w})$ for all $1\leq j\leq n_{k+1}$, it follows that $f_{w}^{n_{k+1}}(V(z))= B(f_{w}^{n_{k+1}}(z),\gamma_{w})$ and that $V(z)= B_{n_{k+1},w}(z,\gamma_w)$.

\medskip
As $F$ is topologically exact on the support of $\mu$, then for  $\hat{w}\in X$, there exists $r_{\hat{w}} > 0$ and $N_{\hat{w}}(\varepsilon)\in \mathbb{N}$, such that for all $(\hat{w},x)\in Supp(\mu)$, we have 
\[
B_{\theta^{j}(\hat{w})}(x, r_{\hat{w}})\subset f_{\hat{w}}^{j}(B_{\hat{w}}(z,\gamma_{\hat{w}}))
\] 

for some $j\leq N_{\hat{w}}(\varepsilon)$.

\medskip
Without loss of generality, we may assume that $n$ satisfies $e^{-4cn}\gamma_{w}< r_{w}$ which gives immediately $V(z)= B_{n_{k+1},w}(z,\gamma_w)\subset B(z,r_{w})$. Thus, we conclude that
\[
 V(z)\subset f_{\theta^{n_{k+1}}(w)}^{j}(B(f_{w}^{n_{k+1}}(z), \gamma_{w}))=f_{\theta^{n_{k+1}}(w)}^{j}(f_{w}^{n_{k+1}}(V(z)))
\]
for some $j\leq N(\varepsilon)$. Put $K(n,(w,z),\varepsilon)=n_{k+1}+j-n$. By Brower's Fixed Point Theorem, we have that
\[
 f_{w}^{n+K(n,(w,z),\varepsilon)}:V(z)\rightarrow f_{w}^{n+K(n,(w,z),\varepsilon)}(V(z))
\]
has a fixed point. Using the nonlacunarity of $(n_k)_k$ we get that 

\begin{align*}
\frac{K(n,(w,z),\varepsilon)}{n}&\leq \frac{n_{k+1}-n+N(\varepsilon)}{n} \\
& \leq\frac{n_{k+1}-n_{k}+N(\varepsilon)}{n_k}\to 0, \ \text{when} \ k\to \infty.
\end{align*}

This finishes the proof of this Lemma.
\end{proof}

\section{Applications}
As consequences of Theorem A, in this section we prove that the measure $\mu \in \mathcal{M}_{\mathbb{P}}^{1}(F)$ ergodic non-uniformly expanding and considering $(\mu_w)_{w\in X}$ its disintegration, then we have that $\mu_w$ can be approximated by measurements supported at return points in the weak$^\ast$ topology. Furthermore, it is shown that the average of the measurements disintegrated along the orbit $(\theta^{n}(w))_{n\geq 0}$  in the dynamics of the base converge in the weak$^{\ast}$ topology to a measure supported at return orbit.


\begin{proof} (of Corollary \ref{corola1})

\begin{enumerate}
    \item [i)] Let $\varphi \in C^0(M)$ and for the compactness of $M$, we have $\displaystyle M = \bigcup_{i=1}^{r} B(y_i,\delta/2)$. There exists a partition $\mathcal{P}$ on $M$, induced by the balls $\{B(y_i,\delta/2): \ 1\leq i \leq r \}$, therefore let $\mathcal{P}=\{P_1, \cdots , P_{r_0}  \}$. On the other hand, give $z\in B(y,\delta) \Leftrightarrow d(y,z)< \delta$, then $\lvert \varphi(y) - \varphi (z)\rvert < \dfrac{\varepsilon}{r_0}$.  Now, using theorem \ref{theoremprincipal}, for every $B(y_i,\delta/2)$, there exist $p_{i,w} \in B(y_i,\delta/2)$, such that $f_w^m(p_i)=p_i$ for some $m\in \mathbb{N}$. It defines $\nu_{w,r_0}= \sum_{i=1}^{r_0} k_{i,w} \delta_{p_i,w}$, where $k_{i,w} = \mu_w(P_i)$ for $1\leq i \leq r_0$. Then,
    
    \begin{align*}
        \Bigl| \int \varphi d\mu_w - \int \varphi d\nu_{w,r_0} \Bigr| & = \Bigl| \sum \int_{P_i} \varphi d\mu_w - \sum k_i \varphi(p_{i,w})  \Bigr| \\
        & = \Bigl| \sum \int_{P_i} (\varphi - \varphi(p_{i,w})) \ d\mu_w \Bigr| \\
        & \leq \sum \int_{P_i} \Bigl| \varphi(y) - \varphi(p_{i,w}) \Bigr| \ d\mu_w \\
        & \leq \sum \dfrac{\varepsilon}{r_0} = \varepsilon.
    \end{align*}
    
    \item [ii)] Let $P_{ret}(f_w) := \{p_w \in M: f^m_w (p_w)=p_w \ \text{for some} \ m\in \mathbb{N}\}$ set of return points of $f_w$. By theorem \ref{theoremprincipal}, we know that $P_{ret}(f_w)$ is dense in $M_w = M$. Then we can define the following map
  $$  
    \begin{array}{cccc}
g \ : & \! X & \! \longrightarrow
& \! M \\
& \! w & \! \longmapsto
& \! g(w)\in P_{ret}(f_w)
\end{array}
$$
such that, if $w_n \to w$ in X, then $g(w_n)\to g(w)$ in $M$. So $g$ is measurable, therefore for $\varphi \in C^0 (M)$ we have that

\[
w\in X \longmapsto \varphi(g(w)) = \int_M \varphi \ d\delta_{g(w)} \in \mathbb{R}
\]
is measurable. 

\item [iii)] Let $\displaystyle \nu = \int \nu_{w,r_0} d\mathbb{P}(x)$ and  $\phi \in L^1_{\mathbb{P}}(X,C^0(M))$, then

\begin{align*}
    \Bigl|\int \phi \ d\mu - \int \phi \ d\nu \Bigr| & = \Bigl| \int_X \int_M \phi_w d\mu_w d\mathbb{P}(w) -  \int_X \int_M \phi_w d\nu_{w,r_0} d\mathbb{P}(w) \Bigr| \\
    & \leq \int_X \Bigl| \int_M \phi_w d\mu_w - \int_M \phi_w d\nu_{w,r_0}  \Bigr| d\mathbb{P}(w)\\
    & \leq \int_X \varepsilon \ d\mathbb{P}(w) = \varepsilon.
\end{align*}

We have that $\nu \in \mathcal{M}_{\mathbb{P}}^{1}(X\times M):=$ space of probability measures on $X\times M$ with marginal $\mathbb{P}$ on $X$.

\end{enumerate}

To prove the next result, we will need the following theorem, where the proof can be found in \cite[Theorem 1.5.8]{Arnold1}

\begin{theorem}[Krylov-Bogolioubov procedure for continuous RDS]
\label{theoremKB-RDS}
Let $F$ be a continuous RDS on a Polish space $M$. Define for on arbitrary $\nu \in \mathcal{M}_{\mathbb{P}}^{1}(X\times M)$ 
\begin{align}
 \label{equationtau} 
\tau_n( \ \ )= \dfrac{1}{n} \sum_{i=0}^{n-1}F(i)\nu( \ \ ).
\end{align}
Then every limit point of $\tau_n$ for $n\rightarrow \infty$ in the topology of weak$^\ast$ convergence is in $\mathcal{M}_{\mathbb{P}}(F)$.
\end{theorem}

We shall show that the measure accumulated by (\ref{equationtau}) is close to the invariant measure $\mu$ in the weak$^\ast$ topology. Indeed, let $\tau\in \mathcal{M}_{\mathbb{P}}(F)$ such that $\tau_n \rightarrow \tau$ in the weak$^\ast$ topology and $\varphi \in C(X\times M)$, then by the invariant of $\mu$ and Corollary \ref{corola1} we have
\[
\Biggl| \int \varphi d F(n)\nu - \int \varphi d\mu \Biggr|= \Biggl| \int \varphi \circ F(n) d\nu - \int \varphi \circ F(n) d \mu \Biggr| < \varepsilon, \quad \forall \ \ n\geq 0,
\]
therefore,

\begin{align*}
    \Biggl|\int \varphi d\tau_n - \int \varphi d\mu   \Biggr| & =  \Biggl|\dfrac{1}{n}\sum_{i=0}^{n-1}\int \varphi d F(i)\nu - \dfrac{1}{n}\sum_{i=0}^{n-1}\int \varphi d\mu   \Biggr| \\
    & \leq \dfrac{1}{n} \sum_{i=0}^{n-1} \Biggl| \int \varphi d F(i)\nu - \int \varphi d \mu   \Biggr| < \varepsilon
\end{align*}
from here,
\[
\Biggl|\int \varphi d\tau - \int \varphi d\mu   \Biggr| < \varepsilon. 
\]
with this we finish the proof.
\end{proof}



\medskip

\begin{proof} (of Corollary \ref{corola2}) 
\begin{enumerate}
    \item 

Suppose that $F:X\times M \to X\times M$ is a system ergodic, $\varphi \in C^{0}(M)$ and let's $\psi \in L_{\mathbb{P}}^{1}(X,C^{0}(M))$ such that $\psi_{w} = \varphi \ \forall \ w\in X$. So, using Birkhoff's Ergodic Theorem for $F$, we get that
\[
\displaystyle \lim_{n\to \infty}\dfrac{1}{n}\sum_{j=0}^{n-1}\varphi(f^{j}_{w}(y))= \int_{X\times M}\varphi d\mu= \int_{X}\int_{M}\varphi d\mu_{w}d\mathbb{P}, \quad \mu-a.e.p.
\]
equivalently, for all $\varepsilon >0$, there exist $n_0 \in \mathbb{N}$ such that $n > n_0$
\[
\left\lvert \dfrac{1}{n}\sum_{j=0}^{n-1}\varphi(f^{j}_{w}(y)) -  \int_{X}\int_{M}\varphi d\mu_{w}d\mathbb{P} \right\rvert < \dfrac{\varepsilon}{2},  \quad \mu-a.e.p..
\]
On the other hand, give $p\in B(n,(w,y),\delta) \Leftrightarrow d(f^{j}_{w}(y), f^{j}_{w}(p))< \delta \quad \forall \ 0\leq j \leq n-1$, then $\lvert \varphi(f^{j}_{w}(y)) - \varphi (f^{j}_{w}(p))\rvert < \dfrac{\varepsilon}{2} \quad \forall \ 0\leq j \leq n-1$. Therefore
\[
\left\lvert \dfrac{1}{n}\sum_{j=0}^{n-1}\varphi(f^{j}_{w}(y)) - \dfrac{1}{n}\sum_{j=0}^{n-1}\varphi(f^{j}_{w}(p))\right\rvert < \dfrac{\varepsilon}{2}.
\]
Now, taking the difference between the Birkhoff's sums, through the orbit of $p$ and using (\ref{specificationproper}), we get
\begin{align*}
       & \left\lvert \dfrac{1}{n+k}\sum_{j=0}^{n+k-1}\varphi(f^{j}_{w}(p)) - \dfrac{1}{n}\sum_{j=0}^{n-1}\varphi(f^{j}_{w}(p))\right\rvert = \\
       & = \left\lvert \dfrac{1}{n+k}\sum_{j=0}^{n+k-1}\varphi(f^{j}_{w}(p)) - \dfrac{1}{n}\sum_{j=0}^{n-1}\varphi(f^{j}_{w}(p)) - \dfrac{1}{n}\sum_{j=n}^{n+k-1}\varphi(f^{j}_{w}(p)) + \dfrac{1}{n}\sum_{j=n}^{n+k-1}\varphi(f^{j}_{w}(p)) \right\rvert\\
    &  = \left\lvert \dfrac{1}{n}\sum_{j=0}^{n+k-1}\varphi(f^{j}_{w}(p)) - \dfrac{1}{n+k}\sum_{j=0}^{n+k-1}\varphi(f^{j}_{w}(p)) - \dfrac{1}{n}\sum_{j=n}^{n+k-1}\varphi(f^{j}_{w}(p)) \right\rvert \\
    & = \left\lvert \dfrac{n+k}{n}\dfrac{1}{n+k}\sum_{j=0}^{n+k-1}\varphi(f^{j}_{w}(p)) - \dfrac{1}{n+k}\sum_{j=0}^{n+k-1}\varphi(f^{j}_{w}(p)) - \dfrac{1}{n}\sum_{j=n}^{n+k-1}\varphi(f^{j}_{w}(p)) \right\rvert \\
    & = \left\lvert \left(\dfrac{n+k}{n}-1\right)\dfrac{1}{n+k}\sum_{j=0}^{n+k-1}\varphi(f^{j}_{w}(p)) - \dfrac{1}{n}\sum_{j=n}^{n+k-1}\varphi(f^{j}_{w}(p)) \right\rvert \\
    & \leq \dfrac{2k}{n}\lVert \varphi \rVert_{\infty} \to 0, \quad \text{when} \ n\to \infty.
\end{align*}

Let's set  $$\displaystyle \mu_{p,w,n+k}=\frac{1}{n+k}\sum_{j=0}^{n+k-1}\delta_{f^{j}_{w}(p)} \ \ \text{with} \ \ f^{n+k}_{w}(p)=p,$$   we obtain
\[
\displaystyle \int_{M} \varphi d\mu_{p,w,n+k} = \dfrac{1}{n+k}\sum_{j=0}^{n+k-1}\varphi(f^{j}_{w}(p)).
\]

Hence

\[
\left\lvert \int_{M} \varphi d\mu_{p,w,n+k} -  \int_{X}\int_{M}\varphi d\mu_{w}d\mathbb{P} \right\rvert < \dfrac{\varepsilon}{2},  \quad \mu-a.e.p..
\]
in other words,

\[
\displaystyle \lim_{n\to \infty} \int_{M}\varphi d\mu_{p,w,n+k} =  \int_{X}\int_{M}\varphi d\mu_{w}d\mathbb{P}.
\]

\item
Being $\mathbb{P}$ an ergodic measure and the function $\displaystyle w\in X \mapsto \int_{M}\varphi d\mu_w \in \mathbb{R}$ mesuarable, for all $\varphi \in C^{0}(M)$. Then, applying Birkhoff's Ergodic Theorem into $\theta:X\to X$, we have
\[
\displaystyle \lim_{n\to \infty}\dfrac{1}{n}\sum_{j=0}^{n-1}\int_{M}\varphi d\mu_{\theta^{j}(w)} = \int_{X}\int_{M}\varphi d\mu_{w}d\mathbb{P}, \quad \mathbb{P}-a.e.p.
\]
moreover, for all $\frac{\varepsilon}{2}>0$, existe $n_0 \in \mathbb{N}$ such that to $n > n_0$
\[
\left \lvert \dfrac{1}{n}\sum_{j=0}^{n-1}\int_{M}\varphi d\mu_{\theta^{j}(w)} -  \int_{M}\varphi d\mu_{p,w,n+k} \right \rvert < \varepsilon, \quad \mu-a.e.p.
\]
equivalently,
\[
\left \lvert \dfrac{1}{n}\sum_{j=0}^{n-1} \mu_{\theta^{j}(w)} -  \mu_{p,w,n+k} \right \rvert < \varepsilon, \quad \mu-a.e.p. 
\]
in the $weak^{\ast}$-topology. 
\end{enumerate}
\end{proof}

\section{Examples}
In this section we will apply Theorem \ref{theoremprincipal} and the Corollaries \ref{corola1} and \ref{corola2} to certain classes of nonuniformly expanding maps. We have that the physical measures obtained in \cite{Araujo} is supported in the averaged sequence of Dirac probability measures along of orbits of return points.
\begin{example}
\label{example1}\normalfont{
	Let $f_{0},f_{1}: M\rightarrow M$ be $C^1$ local diffeomorphisms of a compact and connected manifold $M$ satisfying our conditions (I)-(V). For $1\leq k < \dim M=d$ suppose that $\log \lVert \Lambda^{k} Df_1\rVert < \log \deg f_1$ and consider $$C_{k}(w,x)=\limsup_{n\rightarrow +\infty} \frac{1}{n}\log \|\Lambda^{k}Df_{w}^{n}(x)\|\quad \mbox{and} \quad	C_{k}(w)=\max_{x \in M}C_{k}(w,x).$$
	Let $\mathbb{P}_\alpha$ be the Bernoulli measure on the sequence space $X=\{0,1\}^\mathbb{Z}$ such that $\mathbb{P}_\alpha([1])= \alpha$.
	In  \cite{Alvarez1} was proved the existence of $\alpha\in (0, 1)$ close to $1$ such that
		\begin{eqnarray*}\int\lim_{n\to \infty} \frac{1}{n}\log \lVert \Lambda^{k} Df^{n}_{w}(x)\rVert\ d\mathbb{P}_{\alpha}(w)&<& \alpha \log \deg(f_1) + (1-\alpha)\log \deg (f_0) \\
		&=& \int \log \deg (f_w)\  d\mathbb{P}_{\alpha}(w).
			\end{eqnarray*}
for every $x\in M.$ The rest of the hypotheses (H1)-(H3), follows directly from the fact that functions $f_0, f_1$ are of class $C^1$ and $M$ is compact and connected. }
\end{example}




\begin{example}
Let $M^l$ be a compact $l$-dimensional Riemannian manifold and $\mathcal{D}$ the space of $C^2$ local diffeomorphisms on $M$. Let $(\Omega, T, \mathbb{P})$ be a measure preserving system where $\mathbb{P}$ is ergodic. Define the skew-product by
$$
\begin{array}{cccc}
F\ : & \! \Omega\times M & \! \longrightarrow
& \!\Omega\times M \\
& \! (w,x) & \! \longmapsto	 & \! (T(w),f(w)x)
\end{array}
$$ 
where the maps $f(w)\in \mathcal{D}$ varies continuously on $w\in \Omega$. Fixing positive constants $\delta_0, \delta_1$ small and $p, q \in \mathbb{N}$, satisfying for every $f(w)\in \mathcal{D}$ the following properties:
\begin{enumerate}
    \item [(1)] There exists a covering $B_1,...,B_p,...,B_{p+q}$ of $M$ by injectivity domains s.t.
    \begin{itemize} 
    \item $\|Df(x)^{-1}\|\leq (1+\delta_1)^{-1}$ for every $x\in B_1 \cup \cdots \cup B_p$.
    \item  $\|Df(x)^{-1}\|\leq (1+\delta_0)$ for every $x\in M$.
    \end{itemize}
 \item [(2)] $f$ is everywhere volume expanding: $|\det Df(x)| \geq \sigma_1$ with $\sigma_1 > q$.
 \item [(3)] There exists $A_0$ s.t. $|\log \|f\|_{C^2} | \le A_0$ for any $f\in \mathcal{F} \subset \mathcal{D}$.
\end{enumerate}

 Adding other technical hypotheses, the authors in \cite{Oliveira2} have showed that theses measures are non-uniformly expanding. The generating functions $f_w\in \mathcal{D}$, so you satisfy hypotheses (H1) - (H3).
\end{example}

\begin{example}
\label{example2} As last example, we have the dynamic of Alves and Ara\'ujo in \cite{Araujo}. They work a random perturbation of map $f:M\rightarrow M$ for nonuniformly expanding maps either with or without critical sets. It is considered a continuous map
$$
\Phi: T \rightarrow C^{2}(M,M) \quad \text{as} \quad t\mapsto f_t
$$
from a metric space $T$ into the space of $C^2$ map from $M$ to $M$, with $f= f_{t^*}$ for some fixed $t^* \in T$. Given $x\in M$ the sequence $(f^{n}_{\hat{t}}(x))_{n \geq 1}$ is random orbit of $x$, where $\hat{t}:= (t_1, t_2,...) \in T^{N}$ and
$$
f^{n}_{\hat{t}}= f_{t_n}\circ f_{t_{n-1}}\circ \cdots \circ f_{t_1} \quad \text{for} \quad n \geq 1,
$$
also is taked a family $(\theta_{\varepsilon})_{\varepsilon > 0}$ of probability measures on $T$ such that $(supp \ \theta_\varepsilon)_{\varepsilon >0}$ is nested family of connected compact sets and $supp \ \theta_\varepsilon \rightarrow {t^\ast}$ when $\varepsilon \to 0$. \\
In the context of random perturbations of a map, we say that a Borel probability measure $\mu^\varepsilon$ on $M$ is physical if for a positive Lebesgue measure set of points $x\in M$, the averaged sequence of Dirac probability measures $\delta_{f^{n}_{\hat{t}}(x)}$ along random orbit $(f^{n}_{\hat{t}}(x))_{n \geq 1}$ converge in the weak$^\ast$ topology to $\mu^\varepsilon$ for $\theta^{N}_{\varepsilon}$ almost every $\hat{t}\in T^N$. That is,
$$
\lim_{n\to +\infty} \frac{1}{n}\sum_{j=0}^{n-1}\varphi(f^{j}_{\hat{t}}(x))=\int \varphi d\mu^{\varepsilon} \quad \text{for all} \quad \varphi \in C^{O}(M, \mathbb{R})
$$
and $\theta^{N}_\varepsilon$ almost every $\hat{t}\in T^N$.
Let $F:T^{N}\times M \rightarrow T^{N}\times M$ defined $F(t,x):= (\sigma(\hat{t}), f_{t_1}(x))$, where $\sigma$ is the left shift. In \cite{Araujo} it is proven that for $f:M\to M$ be a $C^2$ nonuniformly expanding local diffeomorphism and $\varepsilon >0$ sufficiently small. Then there are physical measures $\mu_1^\varepsilon,...,\mu_l^\varepsilon$ such that: for each $x\in M$ and $\theta^{N}_{\varepsilon}$ a. e. $\hat{t}\in T^N$, the average of Dirac measures $\delta_{f_{\hat{t}}^{n}(x)}$ converges in weak$^\ast$ topology to some $\mu_i^\varepsilon$ with $1\leq i \leq l$. If $f$ is topologically transitive, then $l=1$.
\medskip
Now, in case particular that $f$ satisfies (I) - (V)  and $\theta^{N}_{\varepsilon}\times \mu$ a expanding measure on $F$, the applying corolario \ref{corola2} we obtain, that some physical measures  $(\mu_i^{\varepsilon})_{1\leq i\leq l}$ is the weak$^{\ast}$ limit of Dirac measures at point of return on the fiber $t^\ast$.
\end{example}

\section*{Acknowledgements}
The author would like to thank K. Oliveira for pointing out this problem  and E. Santana for comments and for useful conversations. We thank also IM-UFAL, Brazil for the hospitality and the opportunity to develop part this work.


\end{document}